\pdfoutput=1
\RequirePackage{ifpdf}
\ifpdf 
\documentclass[pdftex]{sigma}
\else
\documentclass{sigma}
\fi

\numberwithin{equation}{section}

\newtheorem{Theorem}{Theorem}[section]
\newtheorem{Corollary}[Theorem]{Corollary}
\newtheorem{Lemma}[Theorem]{Lemma}
\newtheorem{Proposition}[Theorem]{Proposition}
 { \theoremstyle{definition}
\newtheorem{Definition}[Theorem]{Definition}
\newtheorem{Remark}[Theorem]{Remark} }

\usepackage[all]{xy}
\usepackage{mathtools}
\usepackage{tikz}

\newcommand{\Presym}{\mathsf{Pre}\textrm{-}\mathsf{Sym}}

\newcommand{\hor}{\mathrm{hor}}

\renewcommand{\graph}{\mathrm{graph}}

\newcommand{\id}{\mathrm{id}}

\newcommand{\Lie}{\mathcal{L}}
\newcommand{\T}{\mathbb{T}}

\newcommand{\smooth}{\mathcal{C}}

\newcommand{\V}{\mathbb{V}}
\renewcommand{\t}{\mathfrak{t}}
\newcommand{\MC}{\mathsf{MC}}

\newcommand{\TT}{\ensuremath{\mathbb T}}

\renewcommand{\graph}{\mathrm{graph}}

\newcommand{\oo}{[\![}
\newcommand{\cc}{]\!]}
\newcommand{\lione}{$L_{\infty}[1]$-algebra }

\begin{document}

\allowdisplaybreaks

\newcommand{\arXivNumber}{1807.10148}

\renewcommand{\PaperNumber}{128}

\FirstPageHeading

\ShortArticleName{Deformations of Pre-Symplectic Structures: a Dirac Geometry Approach}

\ArticleName{Deformations of Pre-Symplectic Structures:\\ a Dirac Geometry Approach}

\Author{Florian SCH\"ATZ~$^\dag$ and Marco ZAMBON~$^\ddag$}

\AuthorNameForHeading{F.~Sch\"atz and M.~Zambon}

\Address{$^\dag$~University of Luxembourg, Mathematics Research Unit, Maison du Nombre 6,\\
\hphantom{$^\dag$}~avenue de la Fonte L-4364 Esch-sur-Alzette, Luxembourg}
\EmailD{\href{mailto:florian.schaetz@gmail.com}{florian.schaetz@gmail.com}}

\Address{$^\ddag$~KU Leuven, Department of Mathematics, Celestijnenlaan 200B box 2400,\\
\hphantom{$^\ddag$}~BE-3001 Leuven, Belgium}
\EmailD{\href{mailto:marco.zambon@kuleuven.be}{marco.zambon@kuleuven.be}}

\ArticleDates{Received September 24, 2018, in final form November 27, 2018; Published online December 06, 2018}

\Abstract{We explain the geometric origin of the $L_{\infty}$-algebra controlling deformations of pre-symplectic structures.}

\Keywords{pre-symplectic geometry; deformation theory; Dirac geometry}

\Classification{17B70; 53D17; 58H15}

\section{Introduction}

A pre-symplectic form is just a closed 2-form of constant rank. For instance, the restriction of a~symplectic form to a~coisotropic submanifold (such as the zero level set of a moment map) is pre-symplectic. Given a pre-symplectic from $\eta$ of rank $k$, we constructed in~\cite{SZPre} an algebraic structure that encodes the deformations of $\eta$, i.e., the 2-forms nearby~$\eta$ (in the $C^0$-sense) which are both closed and of constant rank $k$. As in many deformation problems, this algebraic structure is an $L_{\infty}$-algebra, which we call the \emph{Koszul $L_{\infty}$-algebra} of~$\eta$. Its construction~-- which is somewhat involved due to the simultaneous presence of the closedness and constant rank conditions -- relies on a certain $BV_{\infty}$-algebra structure on the differential forms and builds on the work of Fiorenza--Manetti~\cite{Fiorenza--Manetti}. The Koszul $L_{\infty}$-algebra has the property that its Maurer--Cartan elements are in bijection with the pre-symplectic deformations of $\eta$.

Given that pre-symplectic forms are geometric objects, it is natural to ask for a geometric derivation of the algebraic structure that governs their deformations (the Koszul $L_{\infty}$-algebra). The present note provides an answer to this question. The idea is the following: instead of restricting oneself to the realm of 2-forms, work in the larger class of almost Dirac structures, and consider deformations of
 \begin{gather*}\graph(\eta) := \{(v,\eta(v,\cdot)) \, | \, v \in TM \} \subset TM \oplus T^*M\end{gather*}
 within the Dirac structures satisfying a constant rank condition. This is explained in Section~\ref{subsection:Diracview}, which is the heart of this note.

 The first step in \cite{SZPre} is to provide a parametrization of the constant rank forms nearby $\eta$ in terms of (an open subset in) a vector space. This parametrization is obtained naturally by taking the point of view of Dirac linear algebra in Section~\ref{subsection: Dirac-geometric interpretation}.

 The second step in \cite{SZPre} was to show that the closedness condition translates into a Maurer--Cartan equation for a suitable $L_{\infty}$-algebra. In Section~\ref{subsection: integrability via Dirac geometry} we re-obtain this result, and further we improve slightly a result of~\cite{SZPre}, see our Corollary~\ref{corollary: MC for bivector fields}.

Combining these results, in Section~\ref{subsection: presymplintegr} we recover the fact that the $L_{\infty}$-algebra governing deformations of Dirac structures, in the case at hand and upon a suitable restriction, is the Koszul $L_{\infty}$-algebra.

The Koszul $L_{\infty}$-algebra depends on an auxiliary choice of a distribution transverse to $\ker(\eta)$. In the Dirac-geometric interpretation, this translates into a suitable choice of a complement of $\graph(\eta)$ in $TM\oplus T^*M$. One of the achievements of~\cite{GMS} is to establish a general framework to control the effects of changing the complement, exhibiting explicit canonical $L_{\infty}$-isomorphisms between the corresponding $L_{\infty}$-algebras. A consequence of this note and of~\cite{GMS} is that the Koszul $L_{\infty}$-algebra of $(M,\eta)$ is well-defined up to $L_{\infty}$-isomorphisms.

\section{Review: deformations of pre-symplectic structures}\label{section: pre-symplectic structures}

{We review the results on deformations of pre-symplectic structures obtained in the first three sections of~\cite{SZPre}.}

\subsection{Pre-symplectic structures}\label{subsection: pre-symplectic structures}

Fix a smooth manifold $M$.

\begin{Definition}\label{definition: pre-symplectic}A $2$-form $\eta$ on $M$ is called \emph{pre-symplectic} if
\begin{enumerate}\itemsep=0pt
\item[1)] $\eta$ is closed,
\item[2)] the vector bundle map $\eta^\sharp\colon TM \to T^*M, v\mapsto \iota_v\eta = \eta(v,\cdot)$ has constant rank.
\end{enumerate}

A \emph{pre-symplectic manifold} is a pair $(M,\eta)$ consisting of a manifold $M$ and a pre-symplectic structure $\eta$ on $M$. We denote the space of all pre-symplectic structures of rank~$k$ on~$M$ by~$\Presym^k(M)$.
\end{Definition}

A pre-symplectic manifold $(M,\eta)$ gives rise to a distribution
\begin{gather*}K:=\ker\big(\eta^\sharp\big).\end{gather*}
This distribution is involutive since $\eta$ is closed, hence $K$ is tangent to a~foliation of~$M$. Denote by $r\colon \Omega(M) \to \Gamma(\wedge K^*)$ the restriction map. We define the horizontal differential forms as the elements of
\begin{gather*}\Omega_\hor(M):=\ker(r).\end{gather*}
They form a subcomplex of the de Rham complex ${\Omega(M)}$, since the de Rham differential commutes with the pullback of differential forms. The subcomplex $\Omega_\hor(M)$ {is the multiplicative ideal of~$\Omega(M)$ generated by $\Gamma(K^\circ)$, where $K^\circ \subset T^*M$ denotes the annihilator of~$K$.

\subsection{A parametrization of constant rank 2-forms}\label{subsection: Dirac parametrization}

In this subsection we fix a finite-dimensional, real vector space $V$. Recall that a bivector $Z\in \wedge^2 V$ is encoded by the induced} linear map
\begin{gather*}
Z^\sharp\colon \ V^* \to V, \qquad \xi \mapsto \iota_\xi Z = Z(\xi,\cdot).\end{gather*}
Define \begin{gather*} \mathcal{I}_Z:=\big\{\beta \in \wedge^2 V^*\colon \mathrm{id}_V + Z^\sharp \beta^\sharp \text{ is invertible}\big\},\end{gather*}
an open neighborhood of the origin in $\wedge^2 V^*$. Let $F \colon \mathcal{I}_Z\to \wedge^2 V^*$ be the map determined by
\begin{gather}\label{eq:alphabeta}
(F (\beta))^\sharp = \beta^\sharp\big(\id + Z^{\sharp}\beta^\sharp\big)^{-1}.
\end{gather}
The map $F$ is non-linear and smooth. It is a diffeomorphism from $\mathcal{I}_{Z}$ to $\mathcal{I}_{-Z}$, which keeps the origin fixed.

Fix $\eta \in \wedge^2 V^*$ of rank $k$. We now use $F$ to construct submanifold charts for the space $\big({\wedge}^2V^*\big)_k$ of skew-symmetric bilinear forms on $V$ of rank $k$. Fix a subspace $G\subset V$ such that $K\oplus G=V$, where $K=\ker\big(\eta^\sharp\big)$. The restriction of~$\eta$ to~$G$ is a non-degenerate skew bilinear form, therefore there is a unique $Z \in \wedge^2 G \subset\wedge^2 V$ such that
\begin{gather*} Z^\sharp\colon \ G^*\to G, \qquad \xi \mapsto \iota_\xi Z=Z(\xi,\cdot)\end{gather*}
equals $-\big(\eta\vert_G^\sharp\big)^{-1}$.

\begin{Definition}\label{definition: Dirac exponential map}The \emph{Dirac exponential map} $\exp_{\eta}$ of $\eta$ (and for fixed $G$) is the mapping
\begin{gather*}{\exp_{\eta}}\colon \ \mathcal{I}_Z \to \wedge^2 V^*, \qquad \beta \mapsto \eta + F(\beta).\end{gather*}
\end{Definition}

Let $ r\colon {\wedge}^2 V^* \to {\wedge}^2 K^*$ be the restriction map; we have the natural identification $\ker(r) \cong \wedge^2 G^* \oplus (G^*\otimes K^*)$. By the following theorem \cite[Theorem~2.6]{SZPre}, the restriction of $\exp_{\eta}$ {to $\ker(r)$} is a submanifold chart for $\big({\wedge}^2 V^*\big)_k \subset \wedge^2 V^*$.

\begin{Theorem}[parametrizing constant rank forms]\label{theorem: almost Dirac structures}\quad
\begin{enumerate}\itemsep=0pt
\item[$(i)$] Let $\beta \in \mathcal{I}_Z$. Then $\exp_{\eta}(\beta)$ lies in $\big({\wedge}^2 V^*\big)_k$ if, and only if, $\beta$ lies in $\ker(r)=(K^*\otimes G^*)\oplus \wedge^2G^*$.

\item[$(ii)$] Let $\beta = (\mu,\sigma) \in \mathcal{I}_Z\cap \big((K^*\otimes G^*)\oplus \wedge^2G^*\big)$. Then $\exp_{\eta}(\beta)$ is the unique skew-symmetric bilinear form on $V$ with the following properties:
\begin{itemize}\itemsep=0pt
\item its restriction to $G$ equals $(\eta+F(\sigma))|_{\wedge^2G}$;
\item its kernel is the graph of the map $Z^\sharp \mu^\sharp= -\big(\eta\vert_G^{\sharp}\big)^{-1}\mu^\sharp\colon K\to G$.
\end{itemize}

\item[$(iii)$]The Dirac exponential map $\exp_{\eta}\colon \mathcal{I}_Z \to \wedge^2 V^*$ restricts to a~diffeomorphism
\begin{gather*}\mathcal{I}_Z \cap \big((K^*\otimes G^*)\oplus \wedge^2G^*\big) \stackrel{\cong}{\longrightarrow}
\big\{\eta' \in \big({\wedge}^2 V^*\big)_k\,|\, \ker\big({(\eta')^{\sharp}}\big) \text{ is transverse to $G$}\big\}\end{gather*}
onto an open neighborhood of $\eta$ in $\big({\wedge}^2 V^*\big)_k$.
\end{enumerate}
\end{Theorem}

\begin{Remark}\label{rem:parametrize}By the above linear algebra construction, given a pre-symplectic manifold $(M,\eta)$, choosing a subbundle $G$ complementary to $K=\ker\big(\eta^\sharp\big)$, one obtains a map
\begin{gather}\label{eq:expetavb}
\exp_{\eta}\colon \ {\mathcal{I}_Z} \cap \big((K^*\otimes G^*)\oplus \big({\wedge}^2 G^*\big)\big) \to \wedge^2 T^*M.
\end{gather}
It is a not a vector bundle morphism but just a smooth fiberwise map. It maps the zero section to $\eta$, and its image is an open neighborhood of $\eta$ in the space of $2$-forms having the same rank as~$\eta$. The map $\exp_{\eta}$ allows to parametrize deformations of~$\eta$ inside $\Presym^k(M)$ by means of sections $(\mu, \sigma)\in \Gamma(K^*\otimes G^*)\oplus \Gamma\big({\wedge}^2 G^*\big) \cong \Omega_\hor^2(M)$ which are sufficiently small in the $C^0$-sense and for which the 2-form $(\exp_{\eta})(\mu,\sigma)$ is a closed.
\end{Remark}

\subsection[An $L_\infty$-algebra associated to a bivector field]{An $\boldsymbol{L_\infty}$-algebra associated to a bivector field}\label{subsection: Koszul for bivectors}

In this subsection we canonically associate an $L_\infty$-algebra to any bivector field~$Z$ on a manifold~$M$.

\begin{Definition}\label{definition: Koszul bracket of bivector field}Let $Z$ be a bivector field on $M$. The {\em Koszul bracket} associated to $Z$ is the operation
\begin{gather*}
 [\cdot,\cdot]_Z\colon \ \Omega^r(M) \times \Omega^s(M) \to \Omega^{r+s-1}(M),\\
 [\alpha,\beta]_Z:=(-1)^{|\alpha|+1} \big (\Lie_Z(\alpha\wedge \beta)-\Lie_Z(\alpha)\wedge \beta - (-1)^{|\alpha|}\alpha \wedge \Lie_Z(\beta)\big).
 \end{gather*}
\end{Definition}

Here $\Lie_Z = \iota_Z \circ d - d \circ \iota_Z$, where $\iota_Z$ denotes contraction with~$Z$ and~$d$ is the de Rham differential. On $1$-forms $\alpha$ and $\beta$, the Koszul bracket reads $[\alpha,\beta]_Z=\Lie_{Z^{\sharp}\alpha} \beta -\Lie_{Z^{\sharp}\beta} \alpha-d\langle Z, \alpha\wedge \beta\rangle$.

In general the Koszul bracket of $Z$ does not satisfy the graded Jacobi identity (it does only when $Z$ is a Poisson bivector-field). We will see in Proposition~\ref{proposition: extended Koszul} that nevertheless there is a~well-behaved algebraic structure associated to~$Z$. To this aim, recall that a differential form $\alpha \in \Omega^r(M)$ induces by contraction a linear map
\begin{gather*}\alpha^\sharp\colon \ TM \to \wedge^{r-1}T^*M, \qquad v \mapsto \iota_v\alpha,\end{gather*}
and, following \cite[Section~2.3]{FZgeo}, we extend this definition to a collection of forms $\alpha_1,\dots,\alpha_n$ by setting
\begin{align*}
 \alpha_1^\sharp \wedge \cdots \wedge \alpha_n^\sharp\colon \ \wedge^n TM &\to \wedge^{|\alpha_1|+\cdots +|\alpha_n|-n}T^*M, \\
 v_1\wedge \cdots \wedge v_n & \mapsto \sum_{\sigma \in S_n}(-1)^{|\sigma|}\alpha_1^\sharp(v_{\sigma(1)})\wedge \cdots \wedge \alpha_n^\sharp(v_{\sigma(n)}).
\end{align*}

\begin{Definition}\label{definition: trinary bracket}We define the trinary bracket $[\cdot,\cdot,\cdot]_Z\colon \Omega^r(M)\times \Omega^s(M)\times \Omega^k(M)\!\to\! \Omega^{r+s+k-3}(M)$ associated to the bivector field~$Z$
to be
\begin{gather*}[\alpha,\beta,\gamma]_Z:= \big(\alpha^\sharp \wedge \beta^\sharp \wedge \gamma^\sharp\big)\big(\tfrac{1}{2}[Z,Z]\big).\end{gather*}
\end{Definition}

These brackets endow $\Omega(M)[2]$ with an $L_\infty[1]$-algebra structure, extending the results of Fiorenza and Manetti \cite{Fiorenza--ManettiForm}. The following is \cite[Proposition~3.5]{SZPre}:
\begin{Proposition}[the {$L_\infty[1]$}-algebra {$\Omega(M)[2]$}] \label{proposition: extended Koszul} Let $Z$ be a bivector field on $M$. The multilinear maps $\lambda_1$, $\lambda_2$, $\lambda_3$ on the graded vector space $\Omega(M)[2]$ given by
\begin{enumerate}\itemsep=0pt
\item[$1)$] $\lambda_1$ the de~Rham differential $d$,
\item[$2)$] $\lambda_2(\alpha[2] \odot \beta[2])=-\big(\Lie_Z(\alpha\wedge \beta) - \Lie_Z(\alpha)\wedge \beta - (-1)^{\vert \alpha \vert}\alpha \wedge \Lie_Z(\beta)\big)[2] = (-1)^{|\alpha|} ([\alpha,\beta]_Z)[2]$, and
\item[$3)$] $\lambda_3(\alpha[2]\odot \beta[2]\odot \gamma[2])=(-1)^{|\beta|+1}\big(\alpha^\sharp\wedge \beta^\sharp \wedge \gamma^\sharp\big(\frac{1}{2}[Z,Z]\big)\big)[2]$,
\end{enumerate}
define the structure of an $L_\infty[1]$-algebra on $\Omega(M)[2]$.
\end{Proposition}

We now explain the geometric relevance of the $L_{\infty}[1]$-algebra $(\Omega(M)[2],\lambda_1,\lambda_2,\lambda_3)$. As for any $L_\infty[1]$-algebra, it comes with distinguished elements:

\begin{Definition}\label{definition: Maurer--Cartan equation} An element $\beta \in \Omega^2(M)$ is a {\em Maurer--Cartan element} of $(\Omega(M)[2],\lambda_1,\lambda_2,\lambda_3)$ if it satisfies the {\em Maurer--Cartan equation}
\begin{gather*} d(\beta[2]) + \tfrac{1}{2}\lambda_2(\beta[2]\odot \beta[2]) + \tfrac{1}{6}\lambda_3(\beta[2]\odot \beta[2]\odot \beta[2])=0.\end{gather*}
\end{Definition}

Recall that at the beginning of Section~\ref{subsection: Dirac parametrization} we defined an open subset $\mathcal{I}_Z \subset {\wedge^2 T^*M}$ and a~map $F \colon \mathcal{I}_Z\to \wedge^2 T^*M$. The following is \cite[Corollary~3.9]{SZPre}.
\begin{Corollary}[Maurer--Cartan elements of {$\Omega(M)[2]$}]\label{corollary: MC for bivector fields} There is an open subset $\mathcal{U}\subset \mathcal{I}_Z$, which contains the zero section of $\wedge^2 T^*M$, such that a $2$-form $\beta \in \Gamma(\mathcal{U})$ is a Maurer--Cartan element of $(\Omega(M)[2],\lambda_1,\lambda_2,\lambda_3)$ if, and only if, the $2$-form $F(\beta)$ is closed.
\end{Corollary}

In Section \ref{subsection: integrability via Dirac geometry} we will show that for the open subset $\mathcal{U}$ one can choose the whole of~$\mathcal{I}_Z$.

\subsection[The Koszul $L_\infty$-algebra of a pre-symplectic manifold]{The Koszul $\boldsymbol{L_\infty}$-algebra of a pre-symplectic manifold}\label{subsection: Koszul for pre-symplectic}

Let again $\eta$ be a pre-symplectic structure on a manifold~$M$. Fix a subbundle $G\subset TM$ which is complementary to the kernel~$K$ of~$\eta$. Consider the bivector field $Z$ satisfying $Z^\sharp = -\big(\eta\vert_G^\sharp\big)^{-1}$. The following is \cite[Theorem~3.17]{SZPre}.

\begin{Theorem}[the Koszul {$L_\infty[1]$}-algebra]\label{theorem: construction - Koszul L-infty} The $L_{\infty}[1]$-algebra structure on $\Omega(M)[2]$ associated to the bivector field $Z$, see Proposition~{\rm \ref{proposition: extended Koszul}}, maps $\Omega_\hor(M)[2]$ to itself. The subcomplex $\Omega_\hor(M)[2] \subset \Omega(M)[2]$ therefore inherits the structure of an $L_\infty[1]$-algebra, which we call the Koszul $L_\infty[1]$-algebra of~$(M,\eta)$.
\end{Theorem}
 We denote by $\MC(\eta)$ the set of Maurer--Cartan elements of the Koszul \lione of~$(M,\eta)$.

In view of the above theorem, the following result \cite[Theorem~3.19]{SZPre} is an immediate consequence of Theorem~\ref{theorem: almost Dirac structures} and Corollary~\ref{corollary: MC for bivector fields}.

\begin{Theorem}[Maurer--Cartan elements of the Koszul {$L_\infty[1]$}-algebra] \label{theorem: main result} Let $(M,\eta)$ be a~pre-symplectic manifold. The choice of a complement $G$ to the kernel of $\eta$ determines a bivector field~$Z$ by requiring $Z^\sharp=-\big({\eta \vert_G^\sharp}\big)^{-1}$. Suppose $\beta$ is a $2$-form on $M$, which lies in~$\mathcal{I}_Z$. The following statements are equivalent:
\begin{enumerate}\itemsep=0pt
\item[$1.$] $\beta$ is a Maurer--Cartan element of the Koszul \lione $\Omega_\hor(M)[2]$ of $(M,\eta)$, which was introduced in Theorem~{\rm \ref{theorem: construction - Koszul L-infty}}.
\item[$2.$] The image of $\beta$ under the map $\exp_\eta$, which is introduced in Definition~{\rm \ref{definition: Dirac exponential map}}, is a pre-symplectic structure of the same rank as $\eta$.
\end{enumerate}
\end{Theorem}

The above Theorem~\ref{theorem: main result} is the main result of \cite{SZPre}, as it states that the Koszul $L_{\infty}[1]$-algebra governs the deformations of the pre-symplectic structure~$\eta$. More precisely: the fibrewise map~$\exp_\eta$ as in \eqref{eq:expetavb}, on the level of sections, restricts to a map
\begin{gather*}
\boxed{ \exp_\eta\colon \ \Gamma(\mathcal{I}_Z)\cap \MC(\eta) \to \Presym^k(M)}
\end{gather*}
which is injective and whose image consists of the pre-symplectic structures of rank equal to the rank of $\eta$ and with kernel transverse to $G$.

\section{Dirac geometric interpretation}\label{section: Dirac}

In the remainder of this note we explain the geometric framework that underlies the results of Section~\ref{section: pre-symplectic structures} recalled from~\cite{SZPre}. We recover naturally the statements made there and provide some alternative and more geometric proofs.

\subsection{Background on Dirac geometry}
We first review some notions from Dirac linear algebra. Let $V$ be a finite-dimensional, real vector space. We denote by $\mathbb{V}$ the direct sum $V\oplus V^*$ and by $\langle \cdot,\cdot \rangle$ the following non-degenerate pairing on $\V$:
\begin{gather*} \langle (v,\xi),(w,\chi)\rangle := \xi(w) + \chi(v).\end{gather*}

\begin{Definition}A subspace $W\subset \V$ is called {\em Lagrangian} if for all $w,w'\in W$ we have $\langle w,w'\rangle =0$ and $\dim(W)=\dim(V)$. Two subspaces $W$ and $\tilde{W}\subset \V$ are {\em transverse}, if $W\oplus \tilde{W}=\V$.
\end{Definition}

Given an element $Z\in \wedge^2V$, we defined the linear map $Z^\sharp \colon V^* \to V$ in Section~\ref{subsection: Dirac parametrization}, and we can consider the Lagrangian subspace $\mathrm{graph}(Z) :=\big\{\big(Z^\sharp\xi,\xi\big) \, | \, \xi \in V^*\big\}\subset \V$. Similarly, for $\beta \in \wedge^2 V^*$ we define $\beta^\sharp\colon V\to V^*$ and consider $\mathrm{graph}(\beta)$.

Every $\beta \in \wedge^2 V^*$ defines an orthogonal transformation $\t_\beta$ of $(\V,\langle \cdot,\cdot \rangle)$, by
\begin{gather*}(v,\xi)\mapsto \big(v,\xi+\beta^\sharp(v)\big).\end{gather*}
Similarly, every $Z\in \wedge^2 V$ gives rise to an orthogonal transformation $\t_Z$, which takes $(v,\xi)$ to $\big(v + Z^\sharp(\xi),\xi\big)$. In particular, elements of $\wedge^2 V^*$ and $\wedge^2 V$ act on the set of Lagrangian subspaces of $\V$.

\begin{Remark}\label{rem:graphdirac} Suppose $L$, $R$ are transverse Lagrangian subspaces of $\V$. There is a canonical isomorphism \begin{gather*}R\cong L^*, \qquad r\mapsto \langle r, \cdot \rangle|_L.\end{gather*}
Since $R$ is transverse to $L$, any subspace of $\V$ transverse to $R$ is the graph of a linear map $L\to R$. Any \emph{Lagrangian} subspace transverse to $R$ is the graph of a linear map $L\to R$ such that, composing with the canonical isomorphism above, we obtain a \emph{skew-symmetric} linear map $L\to L^*$ (i.e., the sharp map associated to an element of $\wedge^2 L^*$).
\end{Remark}

Let us now briefly recall the basic constituencies of Dirac geometry.
Consider the generalized tangent bundle $\T M=TM \oplus T^*M$. It comes equipped with a non-degenerate pairing
\begin{gather*}\langle(X,\alpha),(Y,\beta)\rangle := \alpha(Y) + \beta(X)\end{gather*}
and the Dorfman bracket
\begin{gather*}\oo(X,\alpha),(Y,\beta)\cc=([X,Y],\Lie_X\beta - \iota_Yd\alpha).\end{gather*}
Together with the projection to $TM$, this makes $\T M$ into an example of Courant algebroid.

\begin{Definition} An \emph{almost Dirac structure} on $M$ is a Lagrangian subbundle $L \subset (\T M,\langle\cdot,\cdot\rangle)$. A~\emph{Dirac structure} is an almost Dirac structure whose space of sections is closed
with respect to the Dorfman bracket $\oo\cdot,\cdot\cc$.
 \end{Definition}

\begin{Remark}\label{rem:LWXapp} Let $L$, $R$ be transverse Dirac structures on $M$. As seen in Remark~\ref{rem:graphdirac}, almost Dirac structures transverse to $R$ are in bijection with elements of~$\Gamma\big({\wedge}^2L^*\big)$. We now recall a~result of Liu--Weinstein--Xu \cite{LWX} establishing when such an almost Dirac structure is Dirac. Recall that every Dirac structure, with the restricted Dorfman bracket and anchor, is a Lie algebroid. Since~$L$ is a Lie algebroid, it induces a differential $d_L$ on $\Gamma(\wedge L^*)$. Further\footnote{The Lie algebroid structures on $L$ and $L^*$ are compatible in the sense that the pair $(L,L^*)$ forms a~Lie bialgebroid.}, since $L^*\cong R$ is a Lie algebroid, it induces a graded Lie bracket $[\cdot,\cdot]_{L^*}$ on $\Gamma(\wedge L^*)[1]$. Together with~$d_L$ and $[\cdot,\cdot]_{L^*}$, the graded vector space $\Gamma(\wedge L^*)[1]$ becomes a differential graded Lie algebra. The main result of \cite{LWX} is: for all $\varepsilon\in \Gamma\big({\wedge}^2L^*\big)$, the graph $L_{\varepsilon}=\{v+\iota_v\varepsilon\colon v\in L\}$ is a Dirac structure iff $\varepsilon$ satisfies the Maurer--Cartan equation, that is
\begin{gather*} d_L\varepsilon + \tfrac{1}{2}[\varepsilon,\varepsilon]_{L^*} =0.\end{gather*}
\end{Remark}

\subsection[Deformations of pre-symplectic structures: the point of view of Dirac geometry]{Deformations of pre-symplectic structures:\\ the point of view of Dirac geometry} \label{subsection:Diracview}

In this subsection we cast the deformations of pre-symplectic forms in the framework of Dirac geometry.

Let $\eta$ be a pre-symplectic form on $M$, with kernel $K$. The natural way to parametrize deformations of $\eta$ is by 2-forms $\alpha$ such that $\eta+\alpha$ is again pre-symplectic, but this parametrization has a serious flaw: the space of such $\alpha$'s does not have a natural vector space structure, due to the constant rank condition. Taking the point of view of Dirac geometry, the above approach to parametrize the deformations of $\eta$ amounts to deforming the Dirac structure $\mathrm{graph}(\eta)$ using $\{0\}\oplus T^*M$ as a complement.

A better way to parametrize the deformations of $\eta$ in terms of Dirac geometry works as follows. Let us first choose a complement $G$ to~$K$. Then
\begin{gather*}G\oplus K^*\end{gather*} is a complement\footnote{Indeed, for every $v\in TM$ we have $\iota_v \eta \in K^{\circ}=G^*$, so requiring that $\iota_v \eta $ lies in $K^*$ implies $\iota_v \eta=0$. This means that $v\in K$, so requiring that $v$ lies in $G$ implies $v=0$.} of $\mathrm{graph}(\eta)$. We can now use $G\oplus K^*$~-- instead of $\{0\}\oplus T^*M$~-- to parametrize deformations of the Dirac structure~$\mathrm{graph}(\eta)$. This choice of complement has the advantage of linearizing the constant rank condition, as we show in Proposition~\ref{prop:rankgood} below. (Notice that when~$\eta$ is {\em symplectic}, the new complement is just $TM$, hence we are deforming $\eta$ by viewing it as a~Poisson structure, just as in~\cite[Section~1.3]{SZPre}.)

We first state two lemmas about the effect of applying the orthogonal transformation $\t_{-\eta}$ of $TM\oplus T^*M$, given by $(v,\xi) \mapsto \big(v,\xi - \eta^\sharp(v)\big)$.

\begin{Lemma}\label{lemma:teZ}Denote by $Z\in \Gamma\big({\wedge}^2 G\big)$ the bivector field such that $Z^{\sharp}$ is the inverse of $-(\eta|_G)^{\sharp}$. Then $\t_{-\eta}$ maps $G\oplus K^*$ to $\mathrm{graph}(Z)$.
\end{Lemma}
\begin{proof}$\t_\eta(\mathrm{graph}(Z))= \big\{\big( Z^\sharp\xi , \xi|_K \big)\colon \xi\in T^*M\big\}=G\oplus K^*$.\end{proof}

Lagrangian subbundles nearby $\mathrm{graph}(\eta)$ can be written, for some $\bar{\beta}\in \Gamma(\wedge^2(\mathrm{graph}(\eta))^*)$, as the graph of the map
\begin{gather*}\bar{\beta}^{\sharp}\colon \ \mathrm{graph}(\eta) \to (\mathrm{graph}(\eta))^*\cong G\oplus K^*,\end{gather*}
by Remark \ref{rem:graphdirac}. We denote this graph as $\Phi_{G\oplus K^*}\big(\bar{\beta}\big)$. Moreover, let $\beta\in \Omega^2(M)$ be the 2-form corresponding to $\bar{\beta}$ under the isomorphism $\mathrm{graph}(\eta)\cong TM, v+\iota_v\eta \mapsto v$ and denote by $\Phi_{Z}({\beta})$ the graph of the map $ \beta^{\sharp}\colon TM\to T^*M\cong \mathrm{graph}(Z)$.

\begin{Lemma}\label{lemma:tme} $\t_{-\eta}$ maps $\Phi_{G\oplus K^*}\big(\bar{\beta}\big)$ to $\Phi_{Z}({\beta})$.
\end{Lemma}
\begin{proof}$\t_{-\eta}$ preserves the pairing on $TM\oplus T^*M$, clearly maps $\mathrm{graph}(\eta)$ to $TM$, and maps $G\oplus K^*$ to $\mathrm{graph}(Z)$ by Lemma \ref{lemma:teZ}. Therefore the statement follows by functoriality.
\end{proof}

Now we can explain why the choice of $G\oplus K^*$ as a complement is a good one to describe pre-symplectic deformations.

\begin{Proposition}\label{prop:rankgood} Let $\bar{\beta}\in \Gamma\big({\wedge}^2(\mathrm{graph}(\eta))^*\big)$.
\begin{enumerate}\itemsep=0pt
\item[$(i)$] The rank of
\begin{gather}\label{eq:cap1}
\Phi_{G\oplus K^*}\big(\bar{\beta}\big)\cap TM
\end{gather}
 equals the rank of
 \begin{gather}\label{eq:cap2}
\{v\in K\colon \iota_v\beta \in G^*\}.
\end{gather}
\item[$(ii)$] Assume that $\Phi_{G\oplus K^*}\big(\bar{\beta}\big)$ is the graph of a $2$-form. Then the rank of this $2$-form equals $\mathrm{rank}({\eta})$ if{f} $\beta$ lies in the vector space $\Omega^2_\hor(M)$ of horizontal $2$-forms.
\end{enumerate}
\end{Proposition}

\begin{proof} (i) Applying the transformation $\t_{-Z}\circ \t_{-\eta}$ to $\Phi_{G\oplus K^*}\big(\bar{\beta}\big)$, by Lemma~\ref{lemma:tme} we obtain $\t_{-Z}(\Phi_{Z}({\beta}))=\mathrm{graph}(\beta)$. Applying it to $TM$ we obtain $\big\{\big(v+Z^\sharp\iota_v\eta, -\iota_v\eta\big)\,|\,v\in {TM} \big\}=K\oplus G^*$.

Hence applying the transformation to the intersection~\eqref{eq:cap1} we obtain
\begin{gather*}\mathrm{graph}(\beta)\cap (K\oplus G^*),\end{gather*} which is isomorphic to \eqref{eq:cap2}.

(ii) Denote by $\eta'$ the 2-form whose graph is $\Phi_{G\oplus K^*}\big(\bar{\beta}\big)$. The kernel of $\eta'$ is given by \eqref{eq:cap1}, and the assertion follows immediately from (i). Recall that the vector space $\Omega^2_{\hor}(M)$ of horizontal 2-forms was defined in Section~\ref{subsection: pre-symplectic structures}, {as the space of 2-forms that vanish on $\wedge^2 K$.}
\end{proof}

\begin{Remark}\label{Rem:twodefproblems}Since $\t_{-\eta}$ is actually an automorphism of the standard Courant algebroid $TM\oplus T^*M$, the following two deformation problems of Dirac structures are equivalent:
 \begin{itemize}\itemsep=0pt
\item deformations of $\mathrm{graph}(\eta)$, using the complement $G\oplus K^*$,
\item deformations of $TM$, using the complement $\mathrm{graph}(Z)$.
\end{itemize}
The latter deformation problem is easier to handle, and the $L_{\infty}[1]$-algebra structure governing it will be recovered in Section~\ref{subsection: integrability via Dirac geometry}.
\end{Remark}

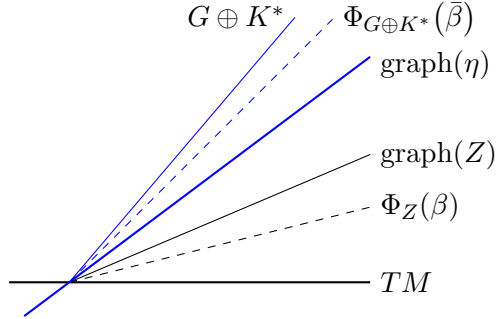
\begin{figure}[h]\centering
\begin{tikzpicture}
\draw[thick] (-0.8,0) -- (4,0); \node [right] at (4,0) {$TM$};
\draw
(0,0) -- (4,1.7); \node [right] at (4,1.7) {graph($Z$)};
\draw[dashed] (0,0) -- (4,1); \node[right] at (4,1) {$\Phi_{Z}({\beta})$};
\draw[thick,blue]
(-0.6,-0.45) -- (4,3); \node [right] at (4,2.9) {graph($\eta$)};
\draw[blue]
(0,0) -- (3,4.7*3/4); \node [left] at (3,4.7*3/4) {$G\oplus K^*$};
\draw[dashed,blue] (0,0) -- (3.5,4*3.5/4); \node[right] at (3.5,4*3.5/4) {$\Phi_{G\oplus K^*}\big(\bar{\beta}\big)$};
\end{tikzpicture}
\caption{The Dirac structures graph($\eta$) and $TM$, together with the complementary Lagrangian subbundles we use to deform them.}
\end{figure}

\subsection{Dirac-geometric interpretation of Section~\ref{subsection: Dirac parametrization}}\label{subsection: Dirac-geometric interpretation}

Using Dirac linear algebra, we explain and re-prove the results recalled in Section~\ref{subsection: Dirac parametrization}.

\subsubsection[Revisiting the map $F$ from formula (\ref{eq:alphabeta})]{Revisiting the map $\boldsymbol{F}$ from formula (\ref{eq:alphabeta})}

Let $V$ be a finite-dimensional, real vector space. We fix a bivector $Z\in \wedge^2V$. Recall that $\mathcal{I}_Z$ consists of elements $\beta\in \wedge^2 V^*$ such that $\mathrm{id} + Z^\sharp \beta^\sharp $ is invertible. In formula~\eqref{eq:alphabeta}, we defined the map $F \colon \mathcal{I}_Z\to \wedge^2 V^*$ given by
\begin{gather*}
(F (\beta))^\sharp = \beta^\sharp\big(\id + Z^{\sharp}\beta^\sharp\big)^{-1}.\end{gather*}
The following lemma provides a geometric explanation of the map $F$.

\begin{Lemma}\label{lemma: almost Dirac structures} Fix $Z\in \wedge^2V$.\samepage
\begin{enumerate}\itemsep=0pt
\item[$(i)$] Taking graphs with respect to the decompositions $\V=V\oplus V^*$ resp.~$\V=V\oplus \mathrm{graph}(Z)$, yields bijections
\begin{align*}
 \Phi_0\colon \wedge^2 V^* &\stackrel{\cong}{\longrightarrow} \{\textrm{Lagrangian subspace of $\V$ transverse to } V^*\},\\
 {\alpha} & \mapsto \{(v,\iota_v{\alpha})\,|\,v\in V\},\\
 \Phi_Z\colon \wedge^2 V^* &\stackrel{\cong}{\longrightarrow} \{\textrm{Lagrangian subspace of $\V$ transverse to } {\mathrm{graph}(Z)}\},\\
 \beta &\mapsto \big\{\big(v+Z^\sharp(\iota_v\beta),\iota_v\beta\big)\,|\,v\in V\big\}.
\end{align*}
\item[$(ii)$] Given $\beta \in \wedge^2 V^*$, the Lagrangian subspace $\Phi_Z(\beta)$ is transverse to $V^*\subset \V$ if, and only if {$\beta\in \mathcal{I}_Z$}.
\item[$(iii)$]
The map \begin{gather*}\Phi_0^{-1}\circ \Phi_Z\colon \ \mathcal{I}_Z\to \wedge^2 V^*\end{gather*} is well-defined and coincides with~$F$.
\end{enumerate}
\end{Lemma}
In particular, the map $F$ is characterized by the property that
\begin{gather}\label{eq:phizbeta}
\mathrm{graph}(F(\beta))=\Phi_Z(\beta)
\end{gather}
 for all $\beta\in \mathcal{I}_Z$. In other words, $F(\beta)$ is obtained taking the graph of $\beta$ w.r.t.\ the splitting $\V=V\oplus \mathrm{graph}(Z)$.

\begin{proof}(i) According to Remark \ref{rem:graphdirac}, any Lagrangian subspace transverse to $V^*$ is the graph of a~skew-symmetric linear map $V\to V^*$, and therefore can be written as $\{(v,\iota_v\alpha)\,|\,v\in V\}$ for some $\alpha\in \wedge^2 V^*$. Similarly, $\mathrm{graph}(Z)$ is transverse to $V$ and the induced isomorphism $\mathrm{graph}(Z)\cong V^*$ is just $(Z^\sharp(\xi),\xi)\mapsto \xi$. Hence any Lagrangian subspace transverse to $\mathrm{graph}(Z)$ can be written as $\{(v,0)+(Z^\sharp(\iota_v\beta),\iota_v\beta)\,|\,v\in V\}$ for some $\beta\in \wedge^2 V^*$.

(ii) The expression for $\Phi_Z(\beta)$ in item (i) shows that $\Phi_Z(\beta)\cap V^* =\{(0,\iota_v\beta)\,|\,v\in V, v+Z^\sharp(\iota_v\beta)=0\}$. This intersection is trivial if{f} $\ker\big(\mathrm{id} + Z^\sharp \beta^\sharp\big) \subseteq \ker\big(\beta^\sharp\big)$. In turn, this condition is equivalent to $\big(\mathrm{id} + Z^\sharp \beta^\sharp\big)$ being injective, and thus invertible.

(iii) Finally, if $\mathrm{id}+Z^\sharp\beta^\sharp$ is invertible, $\Phi_Z(\beta)$ is transverse to $V^*$ by item (ii). By item (i) the element $\Phi_0^{-1}(\Phi_Z(\beta))$ is well-defined. In concrete terms, {it is given by} $\alpha \in \wedge^2 V^*$ such that for all $v\in V$, there is $w \in V$ for which
\begin{gather*}\big(v + Z^\sharp \beta^\sharp(v),\beta^\sharp(v)\big)=\big(w,\alpha^\sharp(w)\big)\end{gather*}
holds. Equivalently, this means that $\alpha^\sharp\big(\mathrm{id}+Z^\sharp\beta^\sharp\big)(v)=\beta^\sharp(v)$ for all $v\in V$. This shows that $\Phi_0^{-1}\circ \Phi_Z$ agrees with~$F$.
\end{proof}

\subsubsection{{Revisiting Theorem~\ref{theorem: almost Dirac structures} (parametrizing constant rank forms)}}

Now let $\eta \in \wedge^2 V^*$ be of rank $k$, fix a complement $G$ to $K:=\ker(\eta)$, and denote by $Z \in \wedge^2 G$ the bivector determined by $ Z^\sharp=-\big(\eta\vert_G^\sharp\big)^{-1}$. In Section~\ref{subsection:Diracview} we considered deformations of the Dirac structure $\graph(\eta)$ using $G\oplus K^*$ as a complement. They are graphs of $2$-forms given by the Dirac exponential map $\exp_{\eta}$ (see Definition~\ref{definition: Dirac exponential map}). More precisely:
\begin{Lemma}\label{lemma:niceeq} For all $\beta\in \mathcal{I}_Z$ we have
\begin{gather}\label{eq:graphexp}
\mathrm{graph}(\exp_{\eta}(\beta))=\Phi_{G\oplus K^*}\big(\bar{\beta}\big).
\end{gather}
\end{Lemma}

 \begin{proof}We have $\mathrm{graph}(\exp_{\eta}(\beta))=\t_\eta(\Phi_Z(\beta))=\Phi_{G\oplus K^*}\big(\bar{\beta}\big)$, where the first equality holds by equation \eqref{eq:phizbeta} and the second by Lemma~\ref{lemma:tme}.
\end{proof}

Using this we recover Theorem~\ref{theorem: almost Dirac structures}, in particular item (i) stating that $\exp_{\eta}(\beta)$ has rank equal to $k=\dim(K)$ if{f} $\beta$ is horizontal.

\begin{proof}[Alternative proof of Theorem~\ref{theorem: almost Dirac structures}] {\sloppy (i)~Apply Proposition~\ref{prop:rankgood}(ii) together with equa\-tion~\eqref{eq:graphexp}.

}

(ii) {We only prove the statement about the kernel of $\exp_{\eta}(\beta)$. Write $\beta = (\mu,\sigma)$. By the proof of Proposition~\ref{prop:rankgood}(i), the intersection of the {subspace~\eqref{eq:graphexp} with $V$} is $(\t_{\eta}\circ \t_{Z})(\mathrm{graph}(\beta)\cap (K\oplus G^*))$, which is precisely the image of $K$ under $\id + Z^\sharp \mu^\sharp $.}

(iii) By Lemma \ref{lemma: almost Dirac structures}(ii), the map $\Phi_Z$ provides a bijection between $\mathcal{I}_Z$ and Lagrangian subspaces transverse to ${\mathrm{graph}(Z)}$ and to $V^*$. Hence $\t_\eta \circ \Phi_Z$
provides a bijection between $\mathcal{I}_Z$ and Lagrangian subspaces transverse to $\t_\eta({\mathrm{graph}(Z)})=G\oplus K^*$ (see Lemma~\ref{lemma:teZ}) and to $V^*$. The latter are exactly the graphs of elements $\eta' \in \wedge^2 V^*$ so that the $\eta'|_{\wedge^2G}$ is non-degenerate. Hence, by the proof of Lemma~\ref{lemma:niceeq}, $\exp_{\eta}$ provides a bijection between $\mathcal{I}_Z$ and such $\eta'$. We conclude using~(i).
\end{proof}

\subsection{Dirac-geometric interpretation of Section~\ref{subsection: Koszul for bivectors}}\label{subsection: integrability via Dirac geometry}

Using Dirac geometry {and adapting results from \cite{FZgeo}}, we explain and re-prove {the} results recalled in Section~\ref{subsection: Koszul for bivectors}. Fix a bivector field $Z$ on~$M$.

\subsubsection[Revisiting Proposition \ref{proposition: extended Koszul} (the {$L_\infty[1]$}-algebra {$\Omega(M)[2]$})]{Revisiting Proposition \ref{proposition: extended Koszul} (the $\boldsymbol{L_\infty[1]}$-algebra $\boldsymbol{\Omega(M)[2]}$)}
In Proposition \ref{proposition: extended Koszul}, the $L_\infty[1]$-algebra $(\Omega(M)[2],\lambda_1,\lambda_2,\lambda_3)$ was constructed out of a bivector field $Z$. It can be recovered using Dirac geometry~-- or more precisely, the deformation theory of Dirac structures~-- as a special case of the construction from \cite[Section~2.2]{FZgeo}.

\begin{Proposition}\label{lem:2.6first} Let $L$ be a Dirac structure and $R$ a complementary almost Dirac structure, i.e., we have a vector bundle decomposition $L\oplus R=\TT M$. Then $\Gamma(\wedge L^*)[2]$ has an induced $L_{\infty}[1]$-algebra structure, whose only non-trivial multibrackets are $\mu_1$, $\mu_2$, $\mu_3$ given as follows:
\begin{enumerate}\itemsep=0pt
\item[$1)$] $\mu_1$ is the differential $d_L$ associated to the Lie algebroid $L$,
\item[$2)$] $\mu_2(\alpha[2] \odot \beta[2])=-(-1)^{|\alpha|}[\alpha,\beta]_{L^*}[2]$, where $[\cdot,\cdot]_{L^*}:=\mathrm{pr}_{R}(\oo \cdot , \cdot \cc)$ denotes the $($extension of$)$ the bracket of the almost Lie algebroid $R\cong L^*$,
\item[$3)$] $\mu_3(\alpha[2]\odot \beta[2]\odot \gamma[2])=(-1)^{|\beta|} \big(\alpha^\sharp\wedge \beta^\sharp \wedge \gamma^\sharp\big) \psi[2]$, where $\psi\in \Gamma\big({\wedge}^3 L\big)$ is given by $\Gamma\big({\wedge}^3 L^*\big)\to \smooth^\infty(M)$, $\xi_1 \wedge \xi_2\wedge \xi_3 \mapsto \langle \mathrm{pr}_{L}(\oo \xi_1 , \xi_2 \cc), \xi_3\rangle$, where we made use of the identification $R \cong L^*$.
\end{enumerate}
\end{Proposition}

More generally, Proposition \ref{lem:2.6first} holds if replacing $\TT M$ by any Courant algebroid.

\begin{proof}The proof is a minor adaptation of the first part of the proof of \cite[Lemma 2.6]{FZgeo}, setting $\varphi=0$ there. We recall briefly the idea of the latter. By \cite{DimaThesis} there is a natural description of the Courant algebroid structure on $\TT M$ in terms of graded geometry. One can use it to apply Voronov's higher derived brackets construction (see \cite{Voronov1,Voronov2}) and obtain an $L_{\infty}[1]$-algebra structure on $\Gamma(\wedge L^*)[2]$. The multibrackets obtained are the ones in the statement of the lemma, as one checks using \cite{DimaThesis} and via computations in local coordinates.
\end{proof}

\begin{proof}[Alternative proof of Proposition \ref{proposition: extended Koszul}] Let $Z$ be a bivector field on $M$. We apply Proposition~\ref{lem:2.6first} for the case $L=TM$ and $R=\mathrm{graph}(Z)$. In this case $d_L$ is the de~Rham differential, and the bracket on $R$ is given by the formula for the Koszul bracket. One checks that $\psi $ is the trivector field $-\frac{1}{2}[Z,Z]$, using \cite[Lemma~1.6]{SZPre}. Hence the $L_{\infty}[1]$-brackets on $\Omega(M)[2]$ given by Proposition~\ref{lem:2.6first} are $\mu_1 =\lambda_1$, $\mu_2 =-\lambda_2$ and $\mu_3 =\lambda_3$. Applying the automorphism $-\mathrm{id}$ to $\Omega(M)[2]$ yields Proposition~\ref{proposition: extended Koszul}.
\end{proof}

 \subsubsection[Revisiting Corollary \ref{corollary: MC for bivector fields} (Maurer--Cartan elements of {$\Omega(M)[2]$})]{Revisiting Corollary~\ref{corollary: MC for bivector fields} (Maurer--Cartan elements of $\boldsymbol{\Omega(M)[2]}$)}
We now turn to Maurer--Cartan elements. In Lemma \ref{lemma: almost Dirac structures}(i), we gave a parametrization of all almost Dirac structures that are transverse to $\mathrm{graph}(Z)$ in terms of $2$-forms $\beta$ on~$M$. This parametrization is given by
\begin{gather*}\beta \mapsto \Phi_Z(\beta)=\big\{\big(v+Z^\sharp(\iota_v\beta),\iota_v\beta\big) \, | \, v \in TM\big\}.\end{gather*}

We present the second part of \cite[Lemma~2.6]{FZgeo}, {which is} an extension of the work by Liu--Weinstein--Xu recalled in Remark~\ref{rem:LWXapp}.
\begin{Proposition}\label{proposition: MC for bivector fields}Let $L$ be given a Dirac structure and~$R$ a complementary almost Dirac structure. An element $\sigma\in \Gamma\big({\wedge}^2 L^*\big)[2]$ is a Maurer--Cartan element of the $L_{\infty}[1]$-algebra structure given in Proposition~{\rm \ref{lem:2.6first}} iff the graph \begin{gather*}\Gamma_{\sigma}:= \{ (X-\iota_X \sigma)\colon X\in L\}\subset L\oplus R\end{gather*}
 is a Dirac structure. $($The above inclusion makes use of the identification $R\cong L^*$.$)$
\end{Proposition}

Corollary \ref{corollary: MC for bivector fields} states that for $\beta \in\Omega^2(M)$ taking values in some sufficiently small neighbor\-hood~$\mathcal{U}$ of the zero section in $\wedge^2 T^*M$ -- in particular taking values in~$\mathcal{I}_Z$, i.e., $\mathrm{id} + Z^\sharp \beta^\sharp$ is invertible,~-- $\beta$ is a Maurer--Cartan element of $(\Omega(M)[2],\lambda_1,\lambda_2,\lambda_3)$ if{f} $F(\beta)$
is closed. We now provide an alternative proof of this result, which also shows that one can choose $\mathcal{U}$ to equal $\mathcal{I}_Z$.

\begin{proof}[Alternative proof of Corollary \ref{corollary: MC for bivector fields}] For {any $\beta\in\Omega^2(M)$}, being a Maurer--Cartan element of the $L_{\infty}[1]$-algebra $(\Omega(M)[2],\lambda_1,\lambda_2,\lambda_3)$ is equivalent to $\Phi_Z(\beta)$ being a Dirac structure. This follows from applying Proposition~\ref{proposition: MC for bivector fields} to the Dirac structure $L=TM$ and to the almost Dirac structure $R=\mathrm{graph}(Z)$, noticing that $\Gamma_{-\beta}=\{(v+Z^\sharp(\iota_v\beta),\iota_v\beta)\,|\,v\in TM\}=\Phi_Z(\beta)$. When $\beta\in\Gamma(\mathcal{I}_Z)$, we know that $\Phi_Z(\beta)$ can be written as the graph of the $2$-form $F(\beta)$, by equation~\eqref{eq:phizbeta}. Now use the fact that the graph of a $2$-form is a Dirac structure if, and only if, the $2$-form is closed.
\end{proof}

\begin{Remark}In this subsection we recovered the $L_\infty[1]$-algebra {$\Omega(M)[2]$} of Proposition~\ref{proposition: extended Koszul} as the $L_\infty[1]$-algebra governing deformations of the Dirac structure $TM$ taking $\graph(Z)$ as a~complement. By Remark~\ref{Rem:twodefproblems}, this deformation problem is equivalent to the deformations of the Dirac structure $\mathrm{graph}(\eta)$ taking $G\oplus K^*$ as the complement. This explains why the $L_\infty[1]$-algebra {$\Omega(M)[2]$} governs the latter deformation problem, and therefore is relevant for the deformations of pre-symplectic structures.
\end{Remark}

\subsection{Dirac-geometric interpretation of Section \ref{subsection: Koszul for pre-symplectic}}\label{subsection: presymplintegr}

Theorem~\ref{theorem: construction - Koszul L-infty} can be deduced from a general statement about (almost) Dirac structures, however doing so amounts essentially to the same computations that were needed for the proof given in~\cite{SZPre}. We include this general statement for the sake of completeness.

\begin{Proposition}In the setting of Proposition~{\rm \ref{lem:2.6first}}, let $K$ be a subbundle of $L$ and define $\Gamma_{\hor}(\wedge L^*)$ as the kernel of the restriction map $\Gamma(\wedge L^*)\to \Gamma(\wedge K^*)$. Then the multibrackets $\mu_1$, $\mu_2$, $\mu_3$ preserve $\Gamma_{\hor}(\wedge L^*)[2]$ if{f} $K$ satisfies the following:
\begin{itemize}\itemsep=0pt
\item $K$ is a Lie subalgebroid of $L$,
\item $\langle \oo \xi_1 , \xi_2 \cc ,\, K+K^{\circ} \rangle=0$ for all $\xi_1 , \xi_2 \in \Gamma(K^{\circ})$, where we use the identification $K^{\circ}\subset L^*\cong R$ and $\oo \cdot,\cdot\cc$ denotes the Dorfman bracket.
\end{itemize}
\end{Proposition}
\begin{proof}We will use the fact that $\mu_1$, $\mu_2$, $\mu_3$ are derivations w.r.t.\ the wedge product in each entry. The Lie algebroid differential $d_L$ preserves $\Gamma_{\hor}(\wedge L^*) $ if{f} the subbundle $K$ is involutive. The bracket $[\cdot,\cdot]_{L^*}$ preserves $\Gamma_{\hor}(\wedge L^*) $ if{f} $\langle \oo \xi_1 , \xi_2 \cc, K\rangle=0$ for all $\xi_1 , \xi_2 \in \Gamma(K^{\circ})$. The trinary bracket $\mu_3$ preserves $\Gamma_{\hor}(\wedge L^*) $ if{f} $\mu_3(\xi_1,\xi_2,\xi_3)=0$ for all $\xi_i \in \Gamma(K^{\circ})$, which in turn is equivalent to $\langle \oo \xi_1 , \xi_2 \cc, \xi_3 \rangle=0$.
\end{proof}

Finally, as mentioned earlier, Theorem~\ref{theorem: main result} follows immediately from the other results presented.

\subsection*{Acknowledgements}We thank Stephane Geudens for comments on a draft of this note. M.Z.~acknowledges partial support by IAP Dygest, the long term structural funding -- Methusalem grant of the Flemish Government,
the FWO under EOS project G0H4518N, the FWO research project G083118N (Belgium).

\pdfbookmark[1]{References}{ref}
\LastPageEnding

\end{document}